\documentclass[12pt]{amsart}
\usepackage{amssymb,amsmath,amsthm}
\numberwithin{equation}{section}
\begin{document}

\theoremstyle{plain}
\newtheorem{theorem}{Theorem}[section]
\newtheorem{lemma}[theorem]{Lemma}
\newtheorem{proposition}[theorem]{Proposition}
\newtheorem{corollary}[theorem]{Corollary}
\newtheorem{conjecture}[theorem]{Conjecture}

\def\mod#1{{\ifmmode\text{\rm\ (mod~$#1$)}
\else\discretionary{}{}{\hbox{ }}\rm(mod~$#1$)\fi}}

\theoremstyle{definition}
\newtheorem*{definition}{Definition}

\theoremstyle{remark}
\newtheorem*{remark}{Remark}
\newtheorem{example}{Example}[section]
\newtheorem*{remarks}{Remarks}

\newcommand{\qq}{{\mathbf Q}}
\newcommand{\rr}{{\mathbf R}}
\newcommand{\nn}{{\mathbf N}}
\newcommand{\zz}{{\mathbf Z}}
\newcommand{\al}{\alpha}
\newcommand{\be}{\beta}
\newcommand{\ga}{\gamma}
\newcommand{\mz}{{\mathcal Z}}
\newcommand{\mi}{{\mathcal I}}
\newcommand{\ep}{\epsilon}
\newcommand{\la}{\lambda}
\newcommand{\de}{\delta}
\newcommand{\De}{\Delta}
\newcommand{\Ga}{\Gamma}
\newcommand{\si}{\sigma}

\title{Laws of Inertia in Higher Degree Binary Forms}

\author{Bruce Reznick}
\address{Department of Mathematics, University of 
Illinois at Urbana-Champaign, Urbana, IL 61801} 
\email{reznick@math.uiuc.edu}
\subjclass[2000]{Primary: 11E76, 15A21}
\begin{abstract}
We consider representations of real forms of even degree as a linear
combination of powers of real linear forms, counting the number of
positive and negative coefficients. We show that the natural
generalization of Sylvester's Law of Inertia holds for binary
quartics, but fails for binary sextics.
\end{abstract}
\date{\today}
\maketitle

\section{Introduction and Overview}

Let $F_{n,d}$ denote the set of real forms $p(x_1,...,x_n)$ of degree
$d$,  $d \ge 1$.  For even $d=2s$, we consider {\it representations}
of $p$ as a linear combination  of the $2s$-th powers of real linear forms:
\begin{equation} \label{E:rep}
p(x_1,\dots,x_n) = \sum_{j=1}^r \la_j\bigl(\al_{j1}x_1 + \dots +
\al_{jn}x_n\bigr)^{2s}, \qquad 0 \neq \la_j \in \mathbb R.
\end{equation}
If there are $a$ positive  (and $b$ negative)
coefficients among the $\la_j$'s, then we  say that \eqref{E:rep} has {\it badge}
$(a,b)$.
Two linear forms are {\it  distinct} if they (and their
$2s$-th powers) are not proportional; a representation is {\it
 honest} if its summands are pairwise distinct.  
Let ${\mathcal B}(p)$ denote the set of possible badges of honest
representations of $p$.   Badges are ordered componentwise: 
$(a,b) \preceq (c,d)$ if $a \le c$ and $b \le d$. Any minimal badge  
 under this ordering is called a {\it signature} of $p$. The set of
 signatures of $p$ is denoted ${\mathcal S}(p)$; in case $p$ has a
 unique signature, we call it {\it the} signature of $p$ and denote it
 $S(p)$. The {\it length} of $p$ is the minimum value of $a+b$ in
 ${\mathcal B}(p)$; it is the shortest possible representation (1.1).
If $q$ is obtained from $p$ by
an invertible linear change of variables, then   ${\mathcal B}(q) =
{\mathcal  B}(p)$.
Identities such as 
\begin{equation*}
x^2 + (x+y)^2 + y^2 = (x-y)^2 + (x+2y)^2 - 3y^2 = 2(x + \tfrac 12 y)^2
+ \tfrac 32 y^2  
\end{equation*}
and
\begin{equation}\label{E:4diff}
(x+2y)^4 - 4(x+y)^4 + 6x^4 - 4(x-y)^4 + (x-2y)^4 = 24y^4
\end{equation}
show that  ${\mathcal B}(p)$ may have more than one element.
We are also interested in sequences $\{p_m\} \subset F_{n,d}$ so that $p_m
\to p$ coefficientwise and $S(p_m) = (a,b)$, but $(a,b) \not\in {\mathcal S}(p)$,
and call this a {\it jump signature sequence}. 

With the exception of a couple of remarks,  we shall restrict
our attention here to binary forms and rename the variables $(x,y)$.
More generally, although the signs of the $\la_j$'s are no longer
relevant, length may be defined over any field which contains
the coefficients of $p$, and for odd degree as well; see \cite{R3} for
much more on this. 

A simple consequence of Sylvester's Law of Inertia (Theorem 2.2) is
that a quadratic form in 
$n$ variables has a unique signature. The main goal of this paper is
to show that binary quartics have unique signatures, but a binary sextic
might not. 

The paper is organized as follows. In section two, we review a number
of useful tools from the literature, many of them due to
J. J. Sylvester. In section three, we prove some general results about
signatures (Theorems 3.1 and 3.2). We find all 
 possible signatures for $p \in F_{2,2s}, s \ge 2$, though not all
 possible sets ${\mathcal S}(p)$.  If $p \neq \pm \ell^{2s}$ is a product of $2s$ real 
linear forms, then it has unique signature $(s,s)$. Associated to $p
\in F_{2,2s}$ is the catalecticant quadratic form $H_p$ in $s+1$ variables.
If $(a,b) \in {\mathcal B}(p)$, then $S(H_p)
\preceq (a,b)$. If $(a,b)$ and $(c,d)$ are both signatures of $p$,
then $a+b+c+d \ge 2s+4$.  
In section four, we show (Theorem 4.1) that binary quartics have
unique signatures, and (Theorem 4.2) if
$p$ has a definite quadratic factor, then $S(p) =
S(H_p)$; otherwise, $S(p) =  (2,2)$. 
 We show that signatures for binary sextics are more
complicated (Theorem 4.4): if  $p(x,y) = x^5y + \al x^3y^3 + xy^5$, $\al \in
(-2,0]$, then ${\mathcal S}(p) = \{(2,3), (3,2)\}$.
Section five contains some open questions and conjectures.

I thank Prof. Krishna Alladi and Prof. Manjul Bhargava and the
participants in the Higher Degree Forms conference at the University
of Florida in May 2009 for their encouragement and support of this
work.

\section{Tools from the literature} 

Every form of degree $d$ is a linear combination of $d$-th powers of linear forms (see
e.g. \cite[p.30]{R1}.) A stronger result holds for binary forms.

\begin{theorem}
Any set  $\{(\al_j x + \be_j y)^d: 0 \le j \le d\}$ of pairwise distinct
$d$-th powers over $\mathbb R$ is linearly independent, and spans $F_{2,d}$.
\end{theorem} 
\begin{proof}
The matrix of this set with respect to  the basis $\{\binom di
x^{d-i}y^i : 0 \le i \le d\}$ is
Vandermonde with determinant $ \prod_{j<k} (\al_j\be_k-\al_k\be_j) \neq 0$. 
\end{proof}

In particular, $p \in F_{2,2s}$ always has a representation $p(x,y) =
\sum_{j=0}^{2s} 
\la_j(x + jy)^{2s}$, where $\la_j \in \mathbb R$. Thus    ${\mathcal B}(p)$
always contains at least one badge $(a,b)$ with  length $\le 2s+1$, and
${\mathcal S}(p)$ is non-empty.

Sylvester Law of Inertia has several different expressions in the
literature. This is the one which motivates this paper.
\begin{theorem}[Sylvester's Law of Inertia]
A real quadratic form in $n$ variables has a unique signature. 
\end{theorem}
\begin{proof}
Suppose $p(x_1,\dots,x_n)$ is a
real quadratic form. Using the standard diagonalization, there is an
invertible linear change of variables after which
\begin{equation} \label{E:diag}
p(x_1,\dots,x_n) = \sum_{j=1}^a x_j^2 - \sum_{j=a+1}^{a+b} x_j^2. 
\end{equation}
Suppose $p$ has another  representation:
\begin{equation} \label{E:alt}
p(x_1,\dots,x_n) = \sum_{j=1}^{c+d} \la_j\bigl(\al_{j1}x_1 + \dots +
\al_{jn}x_n\bigr)^2,
\end{equation}
where for simplicity, $\la_j >0 $ for $ 1\le j \le c$ and $\la_j < 0$
for $c+1 \le j \le c+d$. We claim that $c \ge a$ and $d \ge
b$. Indeed, if $c < a$, then there exists  $(\bar x_1, \dots, \bar x_a)$
 so that
\begin{equation*}
\al_{j1} \bar x_1 + \dots + \al_{ja} \bar x_a = 0,\qquad  1 \le j \le c.
\end{equation*}
The two representations then imply a contradiction:
\begin{equation*}
0 \overset{\eqref{E:diag}}{<}  p(\bar x_1, \dots, \bar x_a,0,\dots,0)
\overset{\eqref{E:alt}}{\le} 0.
\end{equation*}
A similar contradiction follows if $d < b$.
\end{proof}

The diagonalization badge $(a,b)$ from \eqref{E:diag}
 is the signature of the quadratic form $p$, both
in the traditional sense and in our current definition.
No such simple argument applies in higher degree, since 
diagonalization is usually not  possible.

Sylvester proved the next theorem in the course of establishing the
canonical expressions for binary forms. 
The original sources are \cite{G,S1,S2}. Modern discussions and proofs
can be found in \cite{K, KR,R2,R3}.

\begin{theorem}[Sylvester]
Suppose $p(x,y) = \sum_{j=0}^d \binom dj a_j x^{d-j}y^j \in \mathbb C[x,y]$ and 
\begin{equation*}
h(x,y)
= \sum_{j=0}^r 
c_jx^{r-j}y^j = \prod_{k=1}^{r} (\beta_k x - \alpha_k y)
\end{equation*}
 is  a product of pairwise
distinct linear factors over $\mathbb C[x,y]$.   Then there exist
$\lambda_k\in \mathbb C$ so that  
\begin{equation}\label{E:money}
p(x,y) = \sum_{k=1}^r \lambda_k (\alpha_k x + \beta_k y)^d
\end{equation}
if and only if
\begin{equation}\label{E:hankel}
\begin{pmatrix}
a_0 & a_1 & \cdots & a_r \\
a_1 & a_2 & \cdots & a_{r+1}\\
\vdots & \vdots & \ddots & \vdots \\
a_{d-r}& a_{d-r+1} & \cdots & a_d
\end{pmatrix}
\cdot
\begin{pmatrix}
c_0\\c_1\\ \vdots \\ c_r
\end{pmatrix}
=\begin{pmatrix}
0\\0\\ \vdots \\ 0
\end{pmatrix}.
\end{equation}
\end{theorem}

If $p$ and $h$ satisfy these conditions, we say that $h$
is a {\it Sylvester form for $p$}. Theorem 2.3 will only be used in
cases where $a_j,\al_j,\be_j \in \mathbb R$. We remark
that the $\la_k$'s are found by solving \eqref{E:money},  a system of
linear equations with real coefficients: the existence of a solution
over   $\mathbb C$ implies the existence of a solution over $\mathbb R$. 

Another useful classical result is in the 
first article \cite{S3} published in the {\it Proceedings of the
  London Mathematical Society}.  Sylvester proved a conjecture of Newton
regarding the number of complex zeros of certain real
polynomials. This is presented with proof by P\'olya and Szeg\"o
in \cite{PS}[p.48, Prob.79]. 

\begin{theorem}[Sylvester]
Suppose $0 \neq \la_k$ and $\ga_1 < \dots < \ga_r$, $r \ge 2$, are real
numbers such 
that
\begin{equation*}
Q(t) = \sum_{k=1}^r\la_k(t-\ga_k)
\end{equation*}
does not vanish identically. Suppose the sequence
$(\la_1,\dots,\la_r,(-1)^d\la_1)$ has $C$ changes of sign and $Q$ has $Z$
zeros, counting multiplicity. Then $Z \le C$. 
\end{theorem}

We shall apply a homogenized version of this result.

\begin{corollary}
Suppose $p(x,y)$ is a non-zero real form of degree $d$ with $\tau$ real linear
factors (counting multiplicity) and
\begin{equation}  \label{E:sylrep}
p(x,y) = \sum_{k=1}^r \la_k(\cos \theta_k x - \sin \theta_k y)^d,
\end{equation}
where $-\frac {\pi}2 < \theta_1 < \dots < \theta_r \le \frac{\pi}2$, $r
\ge 2$,  and
$\la_k \neq 0$. Suppose there are $\sigma$ sign changes in the sequence
$(\la_1,\la_2,\dots,\la_r,(-1)^d \la_1)$. Then $\tau \le \sigma$.
\end{corollary}

\begin{proof}
We first ``projectivize''  \eqref{E:sylrep} to get the following
representation for $p$:
\begin{equation}  \label{E:sylrep2}
\begin{gathered} 
 \sum_{k=1}^r \tfrac {\la_k}2(\cos \theta_k x
  - \sin \theta_k y)^d + 
\sum_{k=1}^r (-1)^d \tfrac {\la_k}2(\cos (\theta_k +\pi) x - \sin (\theta_k
+ \pi) y)^d .
\end{gathered}
\end{equation}
Observe that the sequence $\frac 12(\la_1,\la_2,\dots,\la_r,(-1)^d \la_1, (-1)^d
\la_2, \dots, (-1)^d \la_r,\la_1)$ has $2\sigma$ sign changes, as does
any cyclic permutation thereof. Thus  in taking any block of $r$ (cyclically)
consecutive summands from \eqref{E:sylrep2} as ``the'' representation
\eqref{E:sylrep}, $\sigma$ will be unchanged. Accordingly, we may make
an invertible change of variables 
\begin{equation} \label{E:dial}
(x,y) \mapsto (\cos \theta\ x + \sin \theta\ y, -\sin \theta\ x + \cos
\theta\ y)
\end{equation}
in \eqref{E:sylrep} to ``dial'' the angles as we wish, without
changing $\tau$ and $\sigma$. We do so to choose $r$ consecutive
angles in the interval $(-\tfrac {\pi}2, \tfrac {\pi}2)$ and so that
$y$ does not divide $p$. 

Having done so, we now
dehomogenize \eqref{E:sylrep} into
\begin{equation*}
Q(t):= p(t,1) = \sum_{k=1}^r \la_k\cos^d \theta_k(t - \tan \theta_k)^d.
\end{equation*}
Since $y \nmid p(x,y)$, there are no ``zeros
at infinity'' and by Theorem 2.4, the number of linear factors of $p$
is equal to the number of 
zeros of $Q$: $\tau = Z \le C = \sigma$.
\end{proof}

For a completely self-contained proof of Corollary 2.5, see
\cite{R3}. Note that both Theorem 2.3 and Corollary 2.5 apply whether
$d$ is even or odd.

\begin{corollary}
If $p(x,y) \neq \pm \ell^{2s}$ splits as a product of $2s$ real linear forms
 and $(a,b) \in {\mathcal B}(p)$, then $(s,s) \preceq (a,b)$.
\end{corollary}
\begin{proof}
By Theorem 2.1, $p$ has at least one representation, necessarily of
length $\ge 2$, which can be put into the form of
\eqref{E:sylrep}. Since there are at least $2s$ sign changes, it must have at least
$s$ positive coefficients and $s$ negative coefficients.
\end{proof}

Corollary 2.6 is improved in Theorem 3.1(2). 
Note that Corollary 2.6 applies to non-standard representations
of $\ell^{2s}$ such as \eqref{E:4diff}.

The final tool involves earlier work  \cite{R1} of the author.
Define
\begin{equation*}
Q_{n,2s}^r = \biggl\{ \sum_{k=1}^r \bigl(\xi_{k1}x_1 + \dots +
\xi_{kn}x_n\bigr)^{2s}: \ \xi_{kj} \in \mathbb R \biggr\},\qquad Q_{n.2s} =
\bigcup_{r=1}^\infty Q_{n,2s}^r. 
\end{equation*}
\begin{theorem}\cite[pp. 36,37]{R1}
For any $(n,2s,r)$, $Q_{n,2s}^r$ is a closed set, and $Q_{n,2s}$ is a
closed convex cone. 
\end{theorem}

The topology here is the usual one, of coefficientwise convergence.
There is a classical inner product on forms under which the
dual cone to $Q_{n,2s}$ is $P_{n,2s}$, the cone of positive
semidefinite forms of degree $2s$ in $n$ variables. 
(See \cite{R1,R2} for an extensive discussion of this inner product.)
Let  $\Sigma_{n,2s}$ denote the cone of sums of squares of polynomials of
degree $s$. Hilbert proved that $\Sigma_{n,2s} = P_{n,2s}$
only for $n=2, 2s=2$ and  $(n,2s) = (3,4)$. In these cases, the dual cone to 
$Q_{n,2s}$ is $\Sigma_{n,2s}$, affording a decisive condition for
membership in $Q_{n,2s}$. 

We restrict our attention to $F_{2,2s}$. We associate to 
$p(x,y) = \sum_{j=0}^{2s} \binom {2s}j a_j x^{2s-j}y^j$ 
the {\it catalecticant} quadratic form
\begin{equation*}
H_p(t) = H_p(t_0,\dots,t_s):= \sum_{i=0}^s \sum_{j=0}^s a_{i+j}t_it_j.
\end{equation*}
(The matrix for $H_p$ corresponds to $(d,r) = (2s,s)$ in \eqref{E:hankel}.) Observe that
\begin{equation} \label{E:catt}
p(x,y) = \sum_{k=1}^r \la_k(\al_k x + \be_k y)^{2s} \iff
H_p(t) = \sum_{k=1}^r \la_k \left(\sum_{i=0}^s  \al_k^{s-i}\be_k^{i} t_i\right)^2.
\end{equation} 

If $f(x,y) = \sum_i\binom di a_i x^{d-i}y^i$ and $g(x,y) =
\sum_i\binom di b_i x^{d-i}y^i$, then the inner product is defined to be
$[f,g] = \sum_i \binom di a_ib_i$.
In particular, if $L(t_0,\dots,t_n;x,y) = \sum_{i=0}^s t_i x^{s-i}y^i$,
then $H_p(t) = [p,L^2]$. 
\begin{lemma}
If $p \in F_{2,2s}$, $\al\de - \be\ga \neq 0$ and $q(x,y) = p(\al x +
\be y, \ga x + \de y)$, then $S(H_q) = S(H_p)$.
\end{lemma}
\begin{proof}
Define $\tilde t_i$ by
\begin{equation*}
\sum_{i=0}^s \tilde t_i x^{s-i}y^i = \sum_{i=0}^s t_i (\al x + \ga
y)^{s-i}(\be x + \de y)^i;
\end{equation*}
the $\tilde t_i$'s are linear functions of the $t_i$'s. Now let
\begin{equation*}
\tilde L(t_0,\dots,t_s;x,y) := L(t_0,\dots,t_s;\al x + \ga y, \be x
+ \de y) = L(\tilde t_0,\dots,\tilde t_n; x,y)
\end{equation*}
 By the contravariant property of the inner product
(\cite[Thm.2.15]{R1}), $H_q(t) = [q,L^2] = [p,\tilde L^2] = H_p(\tilde t)$. 
so that $S(H_q) \preceq S(H_p)$. The change of variables from $p$ to $q$ is 
invertible, so $S(H_p) \preceq S(H_q)$ as well.
\end{proof}

\begin{theorem}\cite[pp.41,61]{R1}
Suppose $p \in F_{2,2s}$ and $rank(H_p) = w$.
\begin{enumerate}
\item We have $p \in Q_{2,2s}$ if and only if $H_p(t)$ is psd.

\item If $p \in Q_{2,2s}$ and $w \le s$, then $p$ can be written
  uniquely as an honest sum of $w$ $2s$-th powers. 

\item If  $p \in Q_{2,2s}$ and $w = s+1$, then for each $(\al, \be)
  \neq (0,0)$, there exists $\la > 0$ so that  $p$ can be written as a
  sum of $s+1$ $2s$-th powers,   one of which is $\la(\al x + \be y)^{2s}$.
\end{enumerate}
 \end{theorem}

If $p \in Q_{2,2s}$, $w=rank(H_p)$ is called the {\it width} of $p$.
Uniqueness in Theorem 2.9(2) is only asserted within the class of representations with
positive coefficients; c.f. \eqref{E:4diff}.
\begin{corollary}
Suppose $p \in F_{2,2s}$ and suppose $S(H_p) =(p_H,n_H)$. 
\begin{enumerate}
\item If  $(a,b) \in {\mathcal B}(p)$, then $(p_H,n_H) \preceq (a,b)$.
\item If $p \in Q_{n,2s}$, then $S(p) = S(H_p) =(p_H,n_H) = (rank(H_p),0)$.  
 \item If  $(a,b) \in {\mathcal B}(p)$ is such that $a+b = p_H + n_H =
   rank(H_p)$,  then the signature of $p$ is $S(p) = S(H_p)$. 
\end{enumerate}
\end{corollary}
\begin{proof}
If $p$ has a representation with badge $(a,b)$, then \eqref{E:catt}
implies that $H_p$ has a representation with the same badge. Quadratic
forms have unique signatures by Theorem 2.1, so (1) is immediate.
If $p \in Q_{2,2s}$, then by Theorem 2.9(2)(3), $(rank(H_p),0) \in
{\mathcal B}(p)$, hence by (1) it is the unique minimal element. 
Finally if $(a,b)
\in {\mathcal B}(p)$ with $a+b = rank(H_p)$, then (1) implies that 
$(a,b) = (r_H,n_H)$ and $(a,b)$ is the unique minimal badge in ${\mathcal B}(p)$.
\end{proof}

The practical significance of Corollary 2.10(3) is that Theorem 2.3
might already show
that the length of $p$ is equal to the rank of $H_p$. In such a case,
we know that the signature of $p$ is $S(H_p)$ without having to
compute an actual representation.

If $s > 1$, then  $(p_H,n_H)$ need not belong to ${\mathcal B}(p)$. 
 By Corollary 2.6, if
$p \neq \pm \ell^{2s}$ splits into $2s$ linear
factors, then any $(a,b) \in{\mathcal B}(p)$ will have $a+b \ge 2s$,
but $r_H +n_H \le s+1$.

We also remark (see \cite[p.124]{R1}) that $(x^2 + y^2)^s$ always has
width $s+1$, and so has signature $(s+1,0)$. Theorem 2.9(3) is
illustrated by the identity 
\begin{equation*}
\binom{2s}s (x^2+y^2)^s = \frac 1{s+1} \sum_{k=0}^s
\left(\cos(\tfrac{k\pi}{s+1}+ \theta) x +  \sin(\tfrac{k\pi}{s+1}+
  \theta) y\right)^{2s}, \qquad \theta \in \mathbb R.
\end{equation*}

Finally, it is easy to construct  jump signature sequences in
$Q_{n,2s}$ and $F_{n,2}$ in which the jump is down. If
\begin{equation*}
f_m = x^{2s} + \tfrac 1m (x^2 + y^2)^s, 
\end{equation*}
then $S(f_m) = (s+1,0)$ and $f_m \to f$, where $S(f) = (1,0)$. 
Since each $Q_{2,2s}^r$ is closed, any jump in $Q_{2,2s}$ will be down.

For quadratic forms, only downward jumps are possible, but every
feasible downward jump occurs. Suppose $a_i,b_i \ge 0$ and $a_1+a_2+b_1+b_2 \le n$. Let
\begin{equation*}
f_m(x_1, \dots, x_n) = \sum_{i=1}^{a_1} x_i^2 + \tfrac 1m
\sum_{i=a_1+1}^{a_1+a_2} x_i^2 - \sum_{i=a_1+a_2+1}^{a_1+a_2+b_1}
x_i^2 - \tfrac 1m \sum_{i=a_1+a_2+b_1+1}^{a_1+a_2+b_1+b_2} x_i^2. 
\end{equation*}
Then $S(f_m) = (a_1+a_2,b_1+b_2)$ and $f_m \to f$ where $S(f) =
(a_1,b_1)$. On the other hand, for $f \in F_{n,2}$, $S(f)$ is the
number of positive and negative roots (counting multiplicity) of
$\phi_f(t)$, the characteristic polynomial of matrix assoicated to $f$. 
Since all roots of $\phi_f(t)$ are real, and since the
roots of a polynomial are continuous functions of its coefficients, it
follows that if $S(f_m) = (a,b)$ and $f_m \to f$, then $S(f) \preceq (a,b)$.

\section{Signatures}

In this section, we determine all possible 
signatures in $F_{2,2s}$ and we find some restrictions on incomparable
signatures. Let $[0,s]^2 = \{(i,j) \in \mathbb Z^2:0 \le i,j \le s\}$.

\begin{theorem}
Fix $s \ge 1$.
\begin{enumerate}
\item  If $p \in F_{2,2s}$, then ${\mathcal B}(p) \subseteq
  \{(s+1,0),(0,s+1)\} \cup [0,s]^2$.  
\item If $p(x,y) \neq \pm \ell^{2s}$ is a product of $2s$ real linear factors,
  then $S(p) = (s,s)$.  
\item If $(u,v) \in \{(s+1,0),(0,s+1)\} \cup [0,s]^2$, then there exists
$p \in F_{2,2s}$ such that $S(p) = (u,v)$.
\end{enumerate}
\end{theorem}

\begin{proof}
Suppose $(a,b) \in {\mathcal S}(p)$ and, specifically, suppose that
\begin{equation}\label{E:pm}
p(x,y) = \sum_{i=1}^a(\al_i x + \be_i y)^{2s} -  \sum_{j=1}^b(\ga_j x
+ \de_j y)^{2s} := p_+(x,y) - p_-(x,y).
\end{equation}
Since  \eqref{E:pm}  is honest and
 $p_+, p_- \in Q_{2,2s}$,  Theorem 2.9 implies that $p_+$
 and $p_-$ can each be written as a sum of  $\le s+1$ $2s$-th
 powers. Since $(a,b)$ is minimal, $a,b \le s+1$.  

We now show that if $(s+1,b) \in {\mathcal S}(p)$, then $b=0$.
If so, by Theorem 2.9(3), we can rewrite $p_+$ as a sum of length
$s+1$, one of whose summands is $\la(\ga_1 x + \de_1 y)^{2s}$. This
cancels with the summand $(\ga_1 x + \de_1 y)^{2s}$ in $p_-$
to yield a representation with a smaller badge. 
A similar contradiction results from $(a,s+1)$ with $a>0$.

For (2), combine Corollary 2.6 and part (1).

For (3), first note that  $p(x,y) = \pm (x^2 + y^2)^s$ has signature $(s+1,0)$
or $(0,s+1)$. We now show that every $(u,v) \in [0,s]^2$ occurs.
If $\psi$ is any product of $2s$ distinct linear
factors, then  $S(\psi) = (s,s)$ by (2), and there exists a representation 
\begin{equation*}
\psi(x,y) = \sum_{j=1}^s(\al_j x + \be_j y)^{2s} - \sum_{j=1}^s (\ga_j x +
\de_j y)^{2s}.
\end{equation*}
We claim that for $0 \le u,v \le s$, 
\begin{equation*}
\psi_{u,v}(x,y) = \sum_{j=1}^u (\al_j x + \be_j y)^{2s} - \sum_{j=1}^v (\ga_j x +
\de_j y)^{2s}
\end{equation*}
has unique signature $(u,v)$. If $(\bar u, \bar v)$ is another badge
for $\psi_{u,v}$ and
\begin{equation*}
\psi_{u,v}(x,y) = \sum_{j=1}^{\bar u} (\bar \al_j x + \bar \be_j y)^{2s} -
\sum_{j=1}^{\bar v} (\bar \ga_j x +\bar \de_j y)^{2s},
\end{equation*}
then we may add the ``missing'' summands to reconstitute $\psi$:
\begin{equation*}
\psi(x,y) =  \sum_{j=1}^{\bar u} (\bar \al_j x + \bar \be_j y)^{2s} +
\sum_{j=u+1}^s (\al_j x + \be_j y)^{2s} - \sum_{j=1}^{\bar v} (\bar \ga_j x
+\bar \de_j y)^{2s} - \sum_{j=v+1}^s (\ga_j x + \de_j y)^{2s}.
\end{equation*}
Since  $(s,s) \preceq (\bar u + s - u, \bar v + s - v)$, we have $(u,v)
\preceq (\bar u, \bar v)$ as claimed.
\end{proof}

\begin{example}
The representation
\begin{equation*}
p(x,y) = 8 x^4 + 48 x^2 y^2 - 8 y^4 = (x + 2y)^4 + 6 x^4 + (x-2y)^4 -
40 y^4 \in F_{2,4}
\end{equation*}
shows that $p$ has badge $(3,1)$, which cannot be minimal by Theorem
3.1(1). Using the argument of the proof, rewrite 
$p_+$  using \eqref{E:4diff} and cancel terms:
\begin{equation*}
p(x,y) =  4(x+y)^4 + 4(x-y)^4 + 24y^4 - 40 y^4 =  4(x+y)^4 + 4(x-y)^4 -16y^4.
\end{equation*}
so $(2,1) \in {\mathcal B}(p)$. Since $H_p(t_0,t_1,t_2) = 8(t_0^2 +
t_0t_2 - t_2^2 + t_1^2)$ has signature $(2,1)$, it follows from Corollary 2.10 that
$S(p) =  (2,1)$.
\end{example} 

We now give a necessary condition for incomparable signatures.

\begin{theorem}
Suppose $p \in F_{2,2s}$ has two  signatures $(a,b)$ and $(c,d)$  such
that $a > c$ and $b < d$. Then $a+d \ge s+3$, $b+c \ge
s+1$,  $\max\{a+b,c+d\} \ge s+2$ and  $a,b,c,d \ge 1$,
\end{theorem}

\begin{proof}
Write
\begin{equation}\label{E:pm2}
\begin{gathered}
p = p_+ - p_- = \sum_{i=1}^a (\al_i x + \be_i y)^{2s} -  \sum_{i=1}^b
(\ga_i x + \de_i y)^{2s}; \\
p = \bar p_+ - \bar p_- = \sum_{i=1}^c (\bar \al_i x + \bar \be_i
y)^{2s} -  \sum_{i=1}^d 
(\bar \ga_i x + \bar \de_i y)^{2s}.
\end{gathered}
\end{equation}
We obtain two representations for $q :=  p_+ + \bar p_- = \bar p_+ + p_-$:
\begin{equation} \label{E:nonmin}
\begin{gathered}
q(x,y) =  \sum_{i=1}^a (\al_i x + \be_i y)^{2s} + \sum_{i=1}^d (\bar \ga_i x
+ \bar \de_i y)^{2s}; \\
q(x,y) =  \sum_{i=1}^b (\ga_i x + \de_i y)^{2s} + \sum_{i=1}^c (\bar
\al_i x + \bar \be_i y)^{2s}. 
\end{gathered}
\end{equation}

Observe that $q \in Q_{2,2s}$, and suppose it has width $w$.  If $w
\le s$, then the two 
representations of $q$ in \eqref{E:nonmin} must be permutations of
each other by Theorem 2.9(2).
 Thus each of the summands $(\al_i x + \be_i y)^{2s}$ must appear in the second
representation. Since the representations in \eqref{E:pm2} are honest,
this summand cannot be one of the $(\ga_i x + \de_i y)^{2s}$'s, and so
must be $(\bar \al_k x + \bar \be_k y)^{2s}$ for some $k$. But the
 $(\al_i x + \be_i y)^{2s}$'s are distinct, as are the 
$(\bar \al_i x + \bar \be_i y)^{2s}$'s. This implies that $a \le c$, a contradiction.

It follows that $w = s+1$, hence $b + c \ge s+1$, and since $a \ge
c+1$ and $d \ge b+1$, we 
have $a + d \ge s+3$ and so $(a+b)+(c+d) = (a+d)+(b+c) \ge
2s+4$. Finally, $a,d \ge 1$ by hypothesis. If $b=0$, then $c \ge s+1$, which
contradicts $a > c$ by Theorem 3.1(1). Assuming $c=0$ leads to a similar contradiction. 
\end{proof}

\section{Quartics and sextics}

We begin this section with an application of the last two theorems.

\begin{theorem}
A binary quartic form has a unique signature.
\end{theorem}
\begin{proof}
By Theorem 3.1(1), the possible  signatures in $F_{2,4}$ are $(0,3)$, $(3,0)$ and
$[0,2]^2$. By Theorem 3.2, if $p$ has two
incomparable signatures 
$(a,b)$ and $(c,d)$, then $a+b+c+d \ge 5+3=8$, so $(a,b) = (c,d) =
(2,2)$, a contradiction.
\end{proof}

The uniqueness allows us to compute the signatures of binary quartics.

\begin{theorem}
If $p \in F_{2,4}$, then either $S(p) = S(H_p)$ or $p \neq \pm \ell^4$
is a product of linear factors  and $S(p) = (2,2)$.
\end{theorem}

\begin{proof}
For clarity we write out $p$ and the matrix for $H_p$:
\begin{equation*}
p(x,y) = a_0x^4 + 4a_1x^3y + 6a_2x^2y^2+4a_3xy^3+a_4y^4, \quad
H_p = \begin{pmatrix}
a_0 & a_1 & a_2\\
a_1 & a_2 & a_3 \\
a_2 & a_3 & a_4 \\
\end{pmatrix}.
\end{equation*}

If $p$ is a product of real linear factors, then either $p = \pm (\al
x + \be y)^4$, in which case $S(p)=S(H_p) = (1,0)$ or (0,1), or else
$S(p) = (2,2)$ by Theorem 3.1(2).

Otherwise, we may assume that $p$ has a definite quadratic
factor. By Lemma 2.8, neither $S(p)$ nor $S(H_p)$ is affected by an
invertible linear change of variables. Therefore, we may diagonalize
this factor and assume that
\begin{equation*}
p(x,y) = (x^2 + y^2)(a x^2 + b x y + c y^2).
\end{equation*}
Now apply (2.7) and, in effect, rotate the axes so that $x^2+y^2$ is
unchanged and $ax^2 + b xy + cy^2$ loses its $xy$-term. After
possible scaling, multiplying  by $-1$ (which flips both $S(p)$ and
$S(H_p)$) and permuting $x$ and $y$, we see that $p$ can be written as
\begin{equation*}
q_u(x,y) = (x^2 + y^2)(x^2 + u y^2) = x^4 + 6 \left(\tfrac{u+1}6\right)
x^2y^2 + y^4, \qquad |u| \le 1.
\end{equation*}
As noted earlier, $S(q_1) = S(H_{q_1}) = (3,0)$, and as $q_{-1}(x,y) =
x^4 - y^4$ and $H_{q_{-1}} =
t_0^2-t_2^2$,  $S(q_{-1}) = S(H_{q_{-1}}) = (1,1)$. Assume 
$|u| < 1$, let $\rho^2 = \frac{u+1}6$ and note that
\begin{equation}\label{E:duh}
q_u(x,y) = \frac 12\biggl( (x + \rho y)^4 +  (x -
 \rho  y)^4 \biggr) +
\left(\frac{-1 +34 u - u^2}{36}\right) y^4.    
\end{equation}
Since $u \in (-1,1)$, \eqref{E:duh} has badge $(3,0)$ if $u > 17 - 12\sqrt 2$,
$(2,0)$ if $u = 17 - 12\sqrt 2$ and $(2,1)$ if    $u < 17 - 12\sqrt 2$.
On the other hand, 
\begin{equation*}
H_{q_u} = \begin{pmatrix}
1  & 0 & \frac{1+u}6\\
0 & \frac{1+u}6 & 0 \\
\frac{1+u}6 & 0 & u \\
\end{pmatrix},
\end{equation*}
and since $\frac{1+u}6 > 0$ for $u \in (-1,1)$, $S(H_q)$ is determined
by the sign of $u - (\frac{1+u}6)^2 = \frac{-1+34 u - u^2}{36}$. In
each case $S(q_u) = S(H_{q_u})$, and this completes the proof.
\end{proof}

The quartic $q_{17-12\sqrt 2}(x,y)$ is less mysterious than it seems: 
\begin{equation*}
\begin{gathered}
x^4 + y^4 = (x^2 + \sqrt 2 x y + y^2)(x^2 - \sqrt 2 x y + y^2) \mapsto\\
 (u^2+v^2)(u^2-4uv+5v^2) \mapsto (w^2+z^2)((3+2\sqrt 2)w^2 +
 (3-2\sqrt 2)z^2)\\ = (3+2\sqrt 2)q_{17-12\sqrt 2}(w,z) 
\end{gathered}
\end{equation*}
under  $(x,y) \mapsto (u-v,\sqrt 2 v)$ 
and $(u,v) \mapsto (a w + b z, - bw + az)$, where $a = \cos\tfrac
{3\pi}8 = \frac14 (2-\sqrt 2)^{1/2}$ and  $b = \sin\tfrac {3\pi}8 =
\frac 14 (2+\sqrt 2)^{1/2}$.
Another approach to real canonical forms for real quartics can be
found in \cite[p.217]{PR}.

\begin{example}
The only possible positive jumps in $F_{2,4}$ occur if $f_m \to f$ and $S(f)
= (2,2)$, since otherwise $S(p) = S(H_p)$. For example, let
\begin{equation*}
f_m(x,y) = \tfrac 1m x^4 + 6x^2y^2 + \tfrac 1m y^4 = 
\tfrac 1m (x^4 + 6m x^2y^2 + y^4)
\end{equation*}
It is easy to see that for $m \ge 1$, $S(f_m) = (2,1)$, but $f_m \to
f$,
where $f(x,y) = x^2y^2$ and $S(f) = (2,2)$. In other words,
$Q_{2,4}^2 - Q_{2,4}^1$ is {\it not} a closed set. 
\end{example}

We turn to sextics. Theorem 3.2 does not rule out incomparable
signatures.

\begin{corollary}
If $p \in S_{2,6}$ and $|{\mathcal S}(p)| \ge 2$, then ${\mathcal
  S}(p) = \{(2,3),(3,2)\}$. 
\end{corollary}

\begin{proof}
If $(a,b)$ and $(c,d)$ are signatures for a sextic and $a > c, b < d$,
then Theorem 3.2 implies that 
$a+d \ge 6$, $b+c \ge 4$ and so $a=d=3$ and $b=c=2$.
\end{proof}

The following theorem analyzes our principal binary sextic example.
\begin{theorem}
Let $q_{\la}(x,y) = 6 x^5 y + 20\la x^3y^3 + 6 x y^5$.
\begin{enumerate}
\item If $\la = 1$, then ${\mathcal S}(q_{\la}) = \{(1,1)\}$
\item If $\la > 0$, $\la \neq 1$, then ${\mathcal S}(q_{\la}) = \{(2,2)\}$.
\item If $-\frac 35 < \la \le 0$, then  ${\mathcal S}(q_{\la}) =
  \{(2,3), (3,2)\}$.
\item If $\la \le -\frac 35$, then ${\mathcal S}(q_{\la}) =
  \{(3,3)\}$.
\end{enumerate}
\end{theorem}

Before we prove this theorem, we state its immediate corollary.
\begin{corollary}
The Law of Inertia fails for binary sextics: there exist $q \in
F_{2,6}$ with two signatures. 
\end{corollary}

The following lemma is a fruitful way of generating forms with two signatures.
\begin{lemma}
If $p(x,-y) = -p(x,y)$, then $(a,b) \in
\mathcal B(p) \implies (b,a) \in \mathcal B(p)$.
\end{lemma}
\begin{proof}
Since $p(x,y) = -p(x,-y)$,
\begin{equation*}
p(x,y) = \sum_{k=1}^r \la_k (\al_k x + \be_k y)^{2s} \implies
p(x,y) = \sum_{k=1}^r -\la_k (\al_k x - \be_k y)^{2s}.
\end{equation*}
\end{proof}

\begin{proof}[Proof of Theorem 4.4]
We first note that $q_{\la}(x,y) = 6xy(x^4 + \frac{10}3 \la x^2y^2 + y^4)$
is a product of six linear factors if $\frac{10}3 \la \le -2$,
establishing case (4). The eigenvalues of 
\begin{equation*}
H_{q_{\la}} = 
\begin{pmatrix}
0 & 1 & 0 & \la \\
1 & 0 & \la & 0 \\
0 & \la & 0 & 1 \\
\la & 0 & 1 & 0
\end{pmatrix}
\end{equation*}
are $\pm(1+\la), \pm(1-\la)$. Thus, if $\la \neq \pm 1$, then
$S(H_{q_{\la}}) = (2,2)$. If $\la = 1$, then $S(H_{q_1}) = (1,1)$ and 
$q_1(x,y) = \frac 12 (x+y)^6 - \frac 12 (x-y)^6$,
establishing case (1). If $\la > -\frac 35$, $\la \neq 1$, then it follows
from Corollary 2.10(3), Theorem 3.1(1) and Lemma 4.6 that ${\mathcal
  S}(q_{\la})$ is $\{(2,2)\}$, 
$\{(2,3),(3,2)\}$ or $\{(3,3)\}$, depending only on the length of $q_{\la}$.

Suppose $\la \neq 1$, $\la > - \frac 35$ and $q_{\la}$ has length
4. Then $h(x,y) = \sum_{i=0}^4 c_i x^{4-i}y^i$ is a Sylvester form for
$q_{\la}$ provided it has four distinct real factors and
\begin{equation*}
\begin{pmatrix}
0 & 1 & 0 & \la & 0\\
1 & 0 & \la & 0 & 1\\
0 & \la & 0 & 1 & 0\\
\end{pmatrix}
\cdot
\begin{pmatrix}
c_0\\c_1\\c_2\\c_3\\c_4
\end{pmatrix}
=\begin{pmatrix}
0\\0\\ 0
\end{pmatrix}.
\end{equation*}
Since $c_1 + \la c_3 = \la c_1 + c_3 = 0$ and $\la^2 \neq 1$, 
$h(x,y) = c_0 x^4 + c_2 x^2y^2 - (c_0 + \la c_2)y^4$. 
If $c_0 = 0$ or $c_0 + \la c_2 = 0$, then $h$ is divisible by $y^2$ or
$x^2$, hence we may scale so that $c_0 = 1$ and assume $h(x,y) =  (x^2
- \ga_1^2y^2)(x^2-\ga_2y^2)$ for some real $\ga_1,\ga_2 \neq 0$. In
this case,
$c_2 = -(\ga_1^2 + \ga_2^2)$ and 
\begin{equation} \label {E:ga}
\ga_1^2\ga_2^2 = -1 + \la(\ga_1^2 +\ga_2^2)
\end{equation}
This is clearly impossible if $\la \le 0$. If $\la > 0$, then
\eqref{E:ga} is equivalent to
\begin{equation*}
\ga_2^2 = \frac {\la \ga_1^2 -1}{\ga_1^2 - \la}.
\end{equation*}
If $\la > 1$, then $\ga_1^2 = 2\la$ implies $\ga_2^2 = 2\la -  1/\la >
0$; if $\la < 1$, then $\ga_1^2 = 2/\la$ implies $\ga_2^2 = \frac{\la}{2-\la^2}$ and
$\ga_1^2 > 2 > \la > \ga_2^2$. In either case, $h$ is a Sylvester
form, so $q_{\la}$ has length 4, and by Corollary 2.10(3), $S(q_{\la})
= (2,2)$. This establishes (2).

In the remaining case, $- \frac 35 < \la \le 0$. We wish to find a quintic
Sylvester form   $h_{\la}(x,y) = \sum_{i=0}^5 c_i x^{5-i}y^i$. Take $0
< u < v$ and let
\begin{equation}\label{E:last}
\begin{gathered}
h_{\la}(x,y) = (x+y)(x^2 + (2 + u) x y + y^2)(x^2 +(2 + v) x y + y^2),\\
6 + u + v + \la (10 + 3 u + 3 v + u v) = 0.
\end{gathered}
\end{equation} 
This satisfies the system
\begin{equation}\label{E:noine}
\begin{pmatrix}
0 & 1 & 0 & \la & 0 & 1\\
1 & 0 & \la & 0 & 1 & 0
\end{pmatrix}
\cdot
\begin{pmatrix}
c_0\\c_1\\c_2\\c_3\\c_4\\c_5
\end{pmatrix}
=\begin{pmatrix}
0\\0
\end{pmatrix}.
\end{equation}
By setting $v = 2u$, we see that \eqref{E:last} holds for
$u = \frac{3 + 5\la}{-2\la}$ and  $v = \frac{3 + 5\la}{-\la}$, so that
$q_{\la}$ for $\la \in (-\frac 35, 0)$ has length five and  $S(q_\la) = \{(2,3),(3,2)\}$.
For $\la = 0$, \eqref{E:last} is impossible for positive $u,v$. In
this case, \eqref{E:noine} becomes $c_0 + c_4 = c_1 + c_5 = 0$. Let
\begin{equation*}
g(x,y) = x(x+y)(x+2y)(x+3y)(x-6y) = x^5 - 25 x^3 y^2 - 60 x^2 y^3 - 36 x y^4
\end{equation*}
and let $h_0(x,y) = g(x,y/\sqrt 6)$. Then $c_0 =1, c_4 = -6^{-2}\cdot 36 =
-1$ and $c_1=c_5=0$, so that $h_0$ satisfies \eqref{E:noine} and $h_0$
has length five, so $S(q_0) = \{(2,3),(3,2)\}$.
\end{proof}

\begin{example}
An exact calculation shows that 
\begin{equation} \label{E:mma}
\begin{gathered}
1296(x+y)^6 - 567(x+2y)^6 +112(x+3y)^6 - (x-6y)^6 - 840x^6 \\
= 3024(x^5 y + 36 x y^5).
\end{gathered}
\end{equation}
If we scale \eqref{E:mma} by sending  $y \mapsto \frac
y{\sqrt 6}$ and divide by $84\sqrt 6$, we obtain a representation of
$q_0(x,y) = 6x^5y + 6xy^5$. 
Observe that $\lim_m q_{1/m} \to q_0$ provides another example of a
positive jump. Also observe that for $\la$ close to 0, neither 
 $S(H_{q_\la})$ nor the zero structure of $q_\la$ change, but
a  jump occurs.
\end{example} 

Larger jumps are possible, as shown in our final example.
\begin{example}
Let $r_\la(x,y) = (x^2-y^2)^3 + 15\la x^2y^2(x^2-y^2)$. Then $r_0(x,y)$
is a product of six linear factors, so $S(r_0) = (3,3)$ and
$r_{1/5}(x,y) = x^6-y^6$, so $S(r_{1/5}) = (1,1)$.  We now show
that if $0 < \la < \frac 15$, then $S(r_\la) = (2,2)$. First, it
is easy to show that $S(H_{r_\la}) = (2,2)$ if $\la \neq 0, -\frac 45$,
  hence the length of $r_\la$ is $\ge 4$  if $\la \in
  (0,\frac 15)$. Second, we produce a Sylvester form of degree four:
 $h(x,y) = \sum_{i=0}^4 c_ix^{4-i}y^i$ must have distinct real factors
 and satisfy
\begin{equation}\label{E:sext}
\begin{pmatrix}
1 & 0 & -\be & 0 & \be\\
0 & -\be & 0 & \be & 0\\
-\be & 0 & \be & 0 & 1\\
\end{pmatrix}
\cdot
\begin{pmatrix}
c_0\\c_1\\c_2\\c_3\\c_4
\end{pmatrix}
=\begin{pmatrix}
0\\0\\ 0
\end{pmatrix},
\end{equation}
where $\be =\frac 15 - \la  \in (0,\frac 15)$. As before, we guess
a solution: 
\begin{equation*}
h(x,y) = (x^2 + (2+u)xy + y^2)(x^2 + (2+v) xy + y^2)
\end{equation*}
with $u\neq v > 0$, and note that \eqref{E:sext} is satisfied if
\begin{equation*}
-1 - \be + \be (6 + 2 u + 2 v + u v) = 0 
\end{equation*}
It can be verified that the choice $u = \frac{1-5\be}{3\be}$ and
$v=\frac{1-5\be}{1+\be}$ meets this criterion. 
\end{example}

\section{Conjectures and Open Questions}

We believe that unique signatures do not exist in $F_{2,2s}$ 
when $s \ge 4$.  Is there an analogue for
Theorem 4.2 in $F_{2,2s}$ when $s \ge 3$? (Even for sextics, this seems to be difficult.)
Are there forms  with signature $(s,s)$ which do not split into a
  product of real forms?
Is the existence of multiple signatures is always a singular
phenomenon, or might they, for example, occur in a   neighborhood?
How can one characterize forms $p$ for which  $S(p) = S(H_p)$? 
Do there exist forms with more than two signatures or with signatures
$(a,b), (c,d)$ for which $a+b \neq  c+d$? (They would have to occur in degree $\ge 8$.) 
What jumps are possible for jump signature sequences? What happens in $F_{3,4}$, the other case in
which catalecticants play a role in determing membership in $Q_{n,2s}$?
Finally,  which properties discussed here are interesting in real closed
  fields besides $\mathbb R$?

%--------------------------------------------------------------------------


\begin{thebibliography}{23}
%--------------------------------------------------------------------------



\bibitem{G}  S. Gundelfinger, \emph{Zur Theorie der bin\"aren Formen},
  J. Reine Angew. Math., \textbf{100} (1886), 413--424. 

\bibitem{K}  J. P. S. Kung, \emph{Gundelfinger's theorem on binary forms},
Stud. Appl. Math., \textbf{75} (1986), 163--169, MR0859177
(87m:11020).

\bibitem{KR}  J. P. S. Kung and G.-C. Rota, \emph{The invariant theory
    of binary forms}, 
 Bull. Amer. Math. Soc. (N. S.), \textbf {10} (1984),  
 27--85, MR0722856 (85g:05002). 


\bibitem{PS}  G. P\'olya and G. Szeg\"o, \emph{Problems and theorems
    in analysis, II}, Springer-Verlag, New York 1976, MR0465631 (57 \#5529).

\bibitem{PR} V. Powers and B. Reznick, \emph{Notes towards a
    constructive proof of Hilbert's Theorem on ternary quartics},
  Proceedings, Quadratic forms and their applications, Dublin 1999
  (A. Ranicki ed.) Contemp. Math., \textbf{272} (2000), 209-227,
  MR1803369   (2001h:11049). 

\bibitem{R1}  B. Reznick, \emph{Sums of even powers of real linear
    forms}, Mem. Amer. Math. Soc. \textbf{96}, 463, 1992, MR1096187
  (93h:11043).

\bibitem{R2}  B. Reznick, \emph{Homogeneous polynomial solutions to
    constant coefficient PDE's over fields}, Adv. Math., \textbf{117} (1996),
  179-192, MR1371648 (97a:12006).


\bibitem{R3} B. Reznick, \emph{The length of binary forms}, in preparation.

\bibitem{S1}  J.J. Sylvester, \emph{An Essay on Canonical Forms,
    Supplement to a Sketch of a Memoir on Elimination, Transformation
    and Canonical Forms}, originally published by George Bell, Fleet
  Street, London, 1851; Paper 34 in \emph{Mathematical
  Papers}, Vol. 1, Chelsea, New York, 1973. Originally published by
Cambridge University Press in 1904. 

\bibitem{S2}   J. J. Sylvester, \emph{On a remarkable discovery in the
    theory of canonical forms and of hyperdeterminants}, originally in
  Phiosophical Magazine, vol. 2, 1851; Paper 42  in \emph{Mathematical
  Papers}, Vol. 1, Chelsea, New York, 1973. Originally published by
Cambridge University Press in 1904.

\bibitem{S3}  J. J. Sylvester, \emph{On an elementary proof and generalization
    of Sir Isaac Newton's hitherto undemonstrated rule for the
    discovery of imaginary roots}, Proc. Lond. Math. Soc. \textbf{1}
  (1865/1866), 1--16; Paper 84  in \emph{Mathematical
  Papers}, Vol.2, Chelsea, New York, 1973. Originally published by
Cambridge University Press in 1908.

\end{thebibliography}
\end{document}